\documentclass[a4, 10pt]{amsart}
\usepackage{amssymb}
\usepackage{amstext}
\usepackage{amsmath}
\usepackage{amscd}
\usepackage{amsthm}
\usepackage{amsfonts}
\usepackage{enumerate}
\usepackage{graphicx}
\usepackage{latexsym}

\theoremstyle{plain}
\newtheorem{thm}{Theorem}[section]
\newtheorem*{thm*}{Theorem}
\newtheorem*{cor*}{Corollary}
\newtheorem*{prop*}{Proposition}

\newtheorem{lem}[thm]{Lemma}

\newtheorem{cor}[thm]{Corollary}

\newtheorem*{claim*}{Claim}

\theoremstyle{definition}

\newtheorem{rem}[thm]{Remark}

\newtheorem*{conj*}{Conjecture}

\newtheorem*{ac}{{\sc Acknowledgments}}

\theoremstyle{remark}

\newtheorem*{tpf}{{\it Proof of Theorem \ref{main}}}

\def\Hom{\operatorname{Hom}}

\def\RHom{\operatorname{{\bf R}Hom}}

\def\p{\mathfrak p}

\def\Z{\Bbb Z}

\def\depth{\operatorname{depth}}

\def\Supp{\operatorname{Supp}}

\def\pd{\operatorname{pd}}
\def\id{\operatorname{id}}

\def\cpd{\operatorname{\text{$C$}-pd}}
\def\cid{\operatorname{\text{$C$}-id}}
\def\opd{\operatorname{\text{$\omega$}-pd}}
\def\oid{\operatorname{\text{$\omega$}-id}}

\def\A{{\mathcal A}}
\def\B{{\mathcal B}}
\def\D{{\mathcal D}}

\begin{document}


\title{A generalization of a theorem of Foxby}
\author{Tokuji Araya}
\address{Nara University of Education, Takabatake-cho, Nara 630-8528, Japan}
\email{araya@math.okayama-u.ac.jp}
\author{Ryo Takahashi}
\address{Department of Mathematical Sciences, Faculty of Science, Shinshu University, 3-1-1 Asahi, Matsumoto, Nagano 390-8621, Japan}
\email{takahasi@math.shinshu-u.ac.jp}
\begin{abstract}
In this paper, it is proved that a commutative noetherian local ring admitting a finitely generated module of finite projective and injective dimensions with respect to a semidualizing module is Gorenstein.
This result recovers a celebrated theorem of Foxby.
\end{abstract}
\maketitle
\section{Introduction}

A semidualizing module over a commutative noetherian ring is a common generalization of a free module of rank one and a dualizing module.
Pioneering studies of semidualizing modules were done by Foxby \cite{F0}, Golod \cite{Go} and Christensen \cite{Cp}.
A semidualizing module not only gives rise to several homological dimensions of modules, but also establishes an equivalence between categories of modules called Auslander and Bass categories.
So far a lot of authors have studied semidualizing modules from various points of view.

Let $R$ be a commutative noetherian ring, and let $C$ be a semidualizing $R$-module.
For a nonzero $R$-module $M$, the {\em $C$-projective dimension} $\cpd_RM$ of $M$ is defined to be the infimum of integers $n$ such that there exists an exact sequence 
$$
0 \to C\otimes_RP_n \to \cdots \to C\otimes_RP_1 \to C\otimes_RP_0 \to M \to 0
$$
of $R$-modules where $P_i$ is projective for $0\le i\le n$.
Dually, the {\em $C$-injective dimension} $\cid_RM$ of $M$ is defined to be the infimum of integers $n$ such that there exists an exact sequence 
$$
0 \to M \to \Hom_R(C,I^0) \to \Hom_R(C,I^1) \to \cdots \to \Hom_R(C,I^n) \to 0
$$
of $R$-modules where $I^i$ is injective for $0\le i\le n$.
The $C$-projective and $C$-injective dimensions of the zero module are defined as $-\infty$.

In the 1970s, Foxby \cite{F}, verifying a conjecture of Vasconcelos \cite{V}, proved that a commutative noetherian local ring is Gorenstein if it admits a nonzero finitely generated module of finite projective and injective dimensions.
As a natural generalization of this statement, Takahashi and White \cite{cdim} asked whether the same conclusion holds true even if ``projective'' and ``injective'' are replaced with ``$C$-projective'' and ``$C$-injective'' respectively.
Recently Sather-Wagstaff and Yassemi \cite{SY} answered that this question has an affirmative answer in the case where the $C$-projective dimension is equal to zero.
The main purpose of this paper is to give a complete answer to the question; we shall prove the following theorem.

\begin{thm}\label{main}
Let $R$ be a commutative noetherian ring, and let $C$ be a semidualizing $R$-module.
Let $M$ be a finitely generated $R$-module with $\cpd_RM<\infty$ and $\cid_RM<\infty$.
Then $R_\p$ is Gorenstein for every $\p\in\Supp_RM$.
\end{thm}

\section{Proof of Theorem \ref{main}}

We denote by $\D(R)$ the derived category of $R$.
To prove our theorem, we give a lemma.

\begin{lem}\label{lem}
Let $R$ be a commutative noetherian ring.
Let $X,Y,Z$ be $R$-complexes.
Assume the following:
\begin{enumerate}[\rm (1)]
\item
$H_i(X)$ and $H_i(Z)$ are finitely generated for all $i\in\Z$, 
\item
$H_i(X)$ and $H_i(Z)$ are zero for all $i\ll 0$,
\item
$\pd_RZ<\infty$.
\end{enumerate}
Then there is a natural isomorphism
$$
\RHom_R(X,Y)\otimes_R^{\bf L}Z\cong\RHom_R(X,Y\otimes_R^{\bf L}Z).
$$
in $\D(R)$.
\end{lem}

\begin{proof}
There exist $R$-complexes
\begin{align*}
P & =(\cdots\to P_{a+1}\to P_a\to 0),\\
Q & =(0\to Q_b\to Q_{b-1}\to\cdots\to Q_{c+1}\to Q_c\to 0)
\end{align*}
isomorphic (in $\D(R)$) to $X$ and $Z$ respectively, such that $P_i$ and $Q_j$ are finitely generated projective $R$-modules for $i\ge a$ and $b\ge j\ge c$.
We have $\RHom_R(X,Y)\otimes_R^{\bf L}Z=\Hom_R(P,Y)\otimes_RQ$ and $\RHom_R(X,Y\otimes_R^{\bf L}Z)=\Hom_R(P,Y\otimes_RQ)$.
There is a natural homomorphism $\Hom_R(P,Y)\otimes_RQ\to\Hom_R(P,Y\otimes_RQ)$ of $R$-complexes; see \cite[(A.2.10)]{Cb}.
This homomorphism is an isomorphism by \cite[(2.7)]{CV}.
\end{proof}

Now we can prove our main theorem.

\begin{tpf}
Replacing $R$ with $R_\p$, we may assume that $R$ is local.
We denote by $k$ the residue field of $R$.
Note from \cite[(2.9)--(2.11)]{cdim} that $M$ is in both the Auslander class $\A_C(R)$ and the Bass class $\B_C(R)$, and that $\Hom_R(C,M)$ (respectively, $C\otimes_RM$) is a nonzero finitely generated $R$-module of finite projective (respectively, injective) dimension.
We have isomorphisms
\begin{align*}
C\otimes_RM & \cong C\otimes_R^{\bf L}M \\
& \cong C\otimes_R^{\bf L}(C\otimes_R^{\bf L}\Hom_R(C,M)) \\
& \cong (C\otimes_R^{\bf L}C)\otimes_R^{\bf L}\Hom_R(C,M)
\end{align*}
in $\D(R)$.
Using Lemma \ref{lem}, we get isomorphisms
\begin{align*}
\RHom_R(k,C\otimes_RM) & \cong \RHom_R(k,(C\otimes_R^{\bf L}C)\otimes_R^{\bf L}\Hom_R(C,M)) \\
& \cong \RHom_R(k,C\otimes_R^{\bf L}C)\otimes_R^{\bf L}\Hom_R(C,M).
\end{align*}
By \cite[(A.7.9)]{Cb}, we obtain:
\begin{align*}
\sup(\RHom_R(k,C\otimes_R^{\bf L}C)) & =\sup(\RHom_R(k,C\otimes_RM))-\sup(k\otimes_R^{\bf L}\Hom_R(C,M)) \\
& =-\depth_R(C\otimes_RM)-\pd_R(\Hom_R(C,M))\in\Z,\\
\inf(\RHom_R(k,C\otimes_R^{\bf L}C)) & =\inf(\RHom_R(k,C\otimes_RM))-\inf(k\otimes_R^{\bf L}\Hom_R(C,M)) \\
& =-\id_R(C\otimes_RM)\in\Z.
\end{align*}
Hence the $R$-complex $\RHom_R(k,C\otimes_R^{\bf L}C)$ is bounded, and so is $C\otimes_R^{\bf L}C$ by \cite[(2.5)]{FI}.
Thus we get $\id_R(C\otimes_R^{\bf L}C)=-\inf(\RHom_R(k,C\otimes_R^{\bf L}C))\in\Z$ by \cite[(A.5.7.4)]{Cb}.
It follows from \cite[(4.4) and (4.6)(a)]{Cp} that there is a natural isomorphism $C\cong\RHom_R(C,C\otimes_R^{\bf L}C)$, and so we have natural isomorphisms $\RHom_R(C\otimes_R^{\bf L}C,C\otimes_R^{\bf L}C)\cong\RHom_R(C,\RHom_R(C,C\otimes_R^{\bf L}C))\cong\RHom_R(C,C)\cong R$.
Therefore $C\otimes_R^{\bf L}C$ is a (semi)dualizing $R$-complex.
It follows from \cite[(3.2)]{FS} that $C$ is isomorphic to $R$.
Thus the dualizing $R$-complex $C\otimes_R^{\bf L}C$ is isomorphic to $R$, which concludes that $R$ is a Gorenstein ring.
\qed
\end{tpf}

\section{Foxby's theorem}

Applying Theorem \ref{main} to the semidualizing $R$-module $C=R$ for a local ring $R$, we immediately recover a well-known result of Foxby.

\begin{cor}\cite[(4.4)]{F}\label{cor}
Let $R$ be a commutative noetherian local ring.
Let $M$ be a nonzero finitely generated $R$-module with $\pd_RM<\infty$ and $\id_RM<\infty$.
Then $R$ is Gorenstein.
\end{cor}

\begin{rem}
\begin{enumerate}[\rm (1)]
\item
Holm \cite{H} investigates existence of modules of finite Gorenstein projective and injective dimensions, and shows a result which also recovers Corollary \ref{cor}.
\item
Our method in the proof of Theorem \ref{main} actually gives a more simple proof of Corollary \ref{cor} than the proof due to Foxby.
In fact, let $R$ and $M$ be as in Corollary \ref{cor}.
Then we have
$$
\RHom_R(k,M)\cong\RHom_R(k,R\otimes_R^{\bf L}M)\cong\RHom_R(k,R)\otimes_R^{\bf L}M,
$$
which gives
\begin{align*}
\id_RR & =-\inf\RHom_R(k,R)\\
& =-\inf\RHom_R(k,M)+\inf(k\otimes_R^{\bf L}M)=\id_RM<\infty,
\end{align*}
namely, $R$ is Gorenstein.
\item
Let $R$ be a Cohen-Macaulay local ring with dualizing module $\omega$.
Then, applying Theorem \ref{main} to the semidualizing $R$-module $C=\omega$, we get the same result as Corollary \ref{cor}, because for an $R$-module $M$ one has $\opd_RM<\infty$ (respectively, $\oid_RM<\infty$) if and only if $\id_RM<\infty$ (respectively, $\pd_RM<\infty$); see \cite[(2.11)]{cdim}.
\end{enumerate}
\end{rem}

\begin{ac}
The authors are indebted to Lars Winther Christensen and Yuji Yoshino for their many useful and helpful comments and suggestions.
The authors also thank Osamu Iyama and the referee for their kind advice and comments.

\end{ac}


\end{document}